\documentclass[11pt]{amsart}
\usepackage[all]{xy}
\usepackage{amscd}
\usepackage{amsfonts}
\usepackage{amsmath}
\usepackage{amssymb}
\usepackage{amsthm}
\usepackage{color}
\usepackage[center]{caption2} 
\usepackage{dsfont}
\usepackage{enumerate}
\usepackage{epic}
\usepackage{float}
\usepackage{geometry}
\usepackage{longtable}

\usepackage{graphicx}
\usepackage{graphics}
\usepackage{hyperref}
\usepackage{indentfirst}
\usepackage{mathrsfs}
\usepackage{multirow}
\usepackage[numbers,sort&compress]{natbib}
\usepackage{pgf}
\allowdisplaybreaks
\numberwithin{equation}{section}
\theoremstyle{plain}
\newtheorem{prop}{Proposition}[section]
\newtheorem{coro}[prop]{Corollary}

\newtheorem{lem}[prop]{Lemma}

\newtheorem{thm}[prop]{Theorem}

\theoremstyle{definition}
\newtheorem{defi}[prop]{Definition}

\newtheorem{exam}[prop]{Example}

\allowdisplaybreaks
\renewcommand\aa{a}
\newcommand\bb{b}

\newcommand\cc{c}

\newcommand\dd{d}
\newcommand\hh{h}

\newcommand\HS[1]{\leavevmode\null\hspace{#1mm}}
\newcommand\Id[1]{\mathsf{Id}(#1)}

\newcommand\ii{i}

\newcounter{ITEM}
\newcommand\ITEM[1]{\setcounter{ITEM}{#1}\leavevmode\hbox{\rm(\roman{ITEM})}}

\newcommand\mm{m}
\newcommand\mA{\mathcal{A}}
\newcommand\mB{\mathcal{B}}

\newcommand\nn{n}

\newcommand\pp{p}
\newcommand\qq{q}

 \newcommand\Span{\mathsf{span}}

\newcommand\wdots{, ...\HS{0.2}, }
\newcommand\xx{x}

\newcommand\yy{y}
\newcommand\zz{z}

\title{Products of commutator ideals of   some Lie-admissible algebras} \thanks{
The work is  supported by 
FCT   UIDB/00212/2020 and UIDP/00212/2020;  by   ``Táýelsizdik urpaqtary" MISD RK} 

\author{Ivan Kaygorodov}
\address{I.K., CMA-UBI, Universidade da Beira Interior, Covilhã, Portugal}
\email{\small kaygorodov.ivan@gmail.com}

\author{Farukh  Mashurov}
\address{F.M., Suleyman  Demirel  University,  Kaskelen,  Kazakhstan; Kazakh-British Technical University, Almaty, Kazakhstan}
\email{\small farukh.mashurov@sdu.edu.kz}

\author{Tran Giang Nam}
\address{T. G. N., Institute of Mathematics, VAST, 18 Hoang Quoc Viet, Cau Giay, Hanoi, Vietnam}
\email{\small tgnam@math.ac.vn}

\author{Zerui Zhang$^{*}$}
\address{Z.Z., School of Mathematical Sciences, South China Normal University, Guangzhou 510631, P. R. China}
\email{\small zeruizhang@scnu.edu.cn}

\thanks{AMS 2020 {Subject Classification}.   17A30, 17B30, 17D25}

\thanks{${}^{*}$ Corresponding author}

\begin{document}

\begin{abstract}
  In this article, we mainly study the products of commutator ideals of Lie-admissible algebras such as Novikov algebras, bicommutative algebras, and assosymmetric algebras. More precisely, we first study the properties of the lower central chains for Novikov algebras and bicommutative algebras. Then we show that for every Lie nilpotent Novikov algebra or Lie nilpotent bicommutative algebra~$\mathcal{A}$, the ideal of~$\mathcal{A}$ generated by the set~$\{ab - ba\mid a, b\in \mathcal{A}\}$ is nilpotent. Finally, we study properties of the lower central chains for assosymmetric algebras, study the products of commutator ideals of assosymmetric algebras and show that the products of commutator ideals have a similar property as that for associative algebras.
\end{abstract}
\maketitle
\section{Introduction}
It is well-known that for an arbitrary algebra~$\mA$, one can always define a multiplication by $[\xx, \yy] = \xx\yy -\yy\xx$, where the juxtaposition denotes multiplication in~$\mA$. And~$\mA$  is called a \emph{Lie-admissible} algebra~\cite{Albert} if~$(\mA, [-,-])$  is a Lie algebra. In this case, we call $(\mA, [-,-])$ the {\it associated Lie algebra}   of  $\mA$.  It is well-known that associative algebras, Novikov algebras, bicommutative algebras and assosymmetric algebras  are Lie-admissible.   We recall that an algebra $\rm B$ over a field $\mathbb F$ is called \emph{bicommutative}~\cite{DIT11} (also known as  \emph{${\rm LR}$-algebras} \cite{bd09}) if it satisfies the identities
\begin{eqnarray}\label{leftcom}
x(yz)&=&y(xz)\ \ \ \ \mbox{(left commutativity)}
\end{eqnarray}
and
\begin{eqnarray}\label{rightcom}
(xy)z&=&(xz)y\ \ \ \ \mbox{(right commutativity)}
\end{eqnarray}
for all $x,y,z\in \rm B$; 
an algebra~$\rm{A}$ over a field $\mathbb F$ is called \emph{assosymmetric}~\cite{Kl57}  if it satisfies the identities
\begin{eqnarray}\label{leftsym}
(xy)z-x(yz)&=&(yx)z-y(xz) \ \ \ \ \mbox{(left symmetry)}
\end{eqnarray}
and
\begin{eqnarray}\label{rightsym}
(xy)z-x(yz)&=&(xz)y-x(zy)\ \ \ \ \mbox{(right symmetry)}
\end{eqnarray}
for all $x,y,z\in \rm A$; and an algebra~$\rm N$ over a field $\mathbb F$ is called a (left) \emph{Novikov algebra}~\cite{GD79,BN85}  if it satisfies the identities~\eqref{rightcom} and~\eqref{leftsym}.

A natural and interesting problem is to determine the structure of a Lie-admissible algebra when its associated Lie algebra has some properties.
Although in general answering this question seems to be a quite difficult task, there have been obtained a number of interesting results regarding the question among which we mention, for example, the following ones. In \cite{jen1}, with a suitable definition of ``commutator ideals", many of the properties of commutator subgroups had analogues in the theory of associative algebras. In particular,  Jennings~\cite{jen1} was interested in extending the notions of ``nilpotent group" and ``solvable group" to rings, and proved that: (1)  if $\mA$ is an associative algebra, not of characteristic $2$ (the latter condition is shown to be indispensable), then $\mA$ is solvable if its associated Lie algebra is;  (2) if $\mA$ is an associative algebra whose associated Lie algebra is nilpotent, then
the ideal~$\mA\circ \mA$  of~$\mA$ generated by the set~$\{ab-ba\mid \aa,\bb\in\mA\}$  is nilpotent. Sharma and Srivastava~\cite{lie-solvable} proved that if $\mA$ is an associative algebra over a field $\mathbb F$ whose associated Lie algebra is
solvable, and if the characteristic of $\mathbb F$ is neither~$2$ nor~$3$, then~$\mA\circ \mA$ is nil. Riley~\cite{Riley} proved that for an associative algebra~$\mA$ over a field of characteristic $p >0$,  the ideal~$\mA\circ \mA$ is nil of bounded index if the associated Lie algebra of $\mA$ is either nilpotent or solvable with $p> 2$.

Novikov algebras, bicommutative algebras, and assosymmetric algebras are important varieties of Lie-admissible algebras. 
Unfortunately, 
there are no new interesting complex finite-dimensional simple algebras in varieties of Novikov, bicommutative and assosymmetric algebras \cite{Kl57,dzhuma,zelm}.
Hence, the study of nilpotent and nearby nilpotent algebras attracts very intensive attention.
So, their algebraic and geometric classifications of nilpotent algebras from these varieties are given in dimension 4 \cite{bi4, ass4, nov4}. Burde and Graaf~\cite{Burde} classified the complex $4$-dimensional Novikov algebras that have a nilpotent associated Lie algebra.  
Pokrass and Rodabaugh proved that each solvable assosymmetric ring of characteristic different from $2$ and $3$  is nilpotent \cite{pr77}.
Filippov proved that a Novikov nil algebra is nilpotent \cite{fil}.
Shestakov and  Zhang studied solvability and nilpotency of Novikov algebras \cite{shest}.
Dzhumadildaev and Tulenbaev obtained a version of the Engel theorem for Novikov algebras in \cite{dzhuma06}.
Burde, Dekimpe, and Vercammen showed that if a nilpotent Lie algebra admits a bicommutative structure, then it admits a complete bicommutative structure, i.e., the right multiplication for the bicommutative structure is always nilpotent \cite{bd09}.  Very recently,  Tulenbaev,  Umirbaev, and Zhelyabin~\cite{TUZ} have studied the Lie solvability of Novikov algebras and showed that over a field
of characteristic $\neq 2$, a Novikov algebra is Lie solvable if and only if its commutator ideal is right nilpotent.   The current paper is a continuation of the investigation of  Lie-admissible algebras such as Novikov algebras, bicommutative algebras, and assosymmetric algebras in terms of their associated Lie algebras. In particular, we characterize these algebras of finite class and the Lie nilpotency of these algebras.

 The article is organized as follows: In Section \ref{sec-central-ideal}, we study central chains of ideals for Novikov and bicommutative algebras and obtain that that if $\mathcal{A}$ is either a Novikov algebra or a bicommutative algebra that is of finite class, then the commutator ideal $\mathcal{A}\circ \mathcal{A}$ is nilpotent (Theorem~\ref{th-pro}). In Section~\ref{sec-Lie nilpotent-nov}, based on Theorem~\ref{th-pro}, we show that a Novikov algebra or bicommutative
algebra $\mathcal{A}$ is Lie nilpotent if and only if $\mathcal{A}$ is of finite class (Theorem~\ref{nil-finclass}). We describes the products of the commutator ideals of Novikov algebras and bicommutative algebras that does not hold for associative algebras in general (Theorem~\ref{prod-com-id}). In Section~\ref{sec-4},
we generalizes some properties of associative algebras for assosymmetric algebras. In particular, we prove that if $\mathcal{A}$ be an assosymmetric algebra of finite class, then $\mathcal{A}\circ \mathcal{A}$ is nilpotent of nilpotent index less or equal to the class of $\mathcal{A}$ (Theorem~\ref{thm45}), and show that~$Id(\mathcal{A}_{[i]}) Id(\mathcal{A}_{[j]})\subseteq Id(\mathcal{A}_{[i+j-1]})$ if~$i$ or $j$ is odd
for every assosymmetric algebra $\mathcal{A}$ (Theorem~\ref{cp-ass}).

\section{Central chains of ideals for Novikov and bicommutative algebras }\label{sec-central-ideal}
The aim of this section is to show that some properties of lower central chain of associative algebras also hold for Novikov algebras and bicommutative algebras. These properties will be very useful for the study of Lie nilpotent Novikov algebras and Lie nilpotent bicommutative algebras in the next section.  We begin with some basic facts on Lie-admissible algebras.

Let $\mA$ be an arbitrary Lie-admissible algebra over a given field $F$. We define
$$[\aa,\bb]=\aa\bb-\bb\aa$$ for all $a$ and $b\in\mA$.
For all subspaces $A$, $B$, $C$ of~$\mA$, we define
$$[A,B]=\Span\{[\aa,\bb] \mid \aa\in A, \bb\in B\},   \   \ AB=\Span\{\aa\bb\mid \aa\in A, \bb\in B\}$$
and
$$(A,B,C)=\Span\{(\aa,\bb,\cc) \mid \aa\in A, \bb\in B, \cc\in C\},$$
where the associator~$(a,b,c)$ means~$(ab)c-a(bc)$.
We call a space~$V\subseteq \mA$ a \emph{Lie ideal} of~$\mA$  if we have~$[V,\mA]\subseteq V$.
Finally, for all subspaces  $A$ and $B$ of~$\mA$, we define
$$A\circ B=\Id{[A,B]},$$
that is, the ideal of~$\mA$ generated by~$[A,B]$. Following the idea of Jennings \cite{jen1}, we call~$A\circ B$ the  \emph{commutator ideal} of~$A$ and~$B$.  We clearly have~$A\circ B=B\circ A$.

Equipped with the notion of commutator ideals, we are now able to recall the notion of central chains of ideals of a Lie-admissible algebra~$\mA$.

Let
\begin{equation}\label{def-cent}
\mA=A_1\supseteq A_2 \supseteq \cdots \supseteq A_\mm
\supseteq A_{\mm+1}=(0)
\end{equation}
be a chain of ideals of~$\mA$. Such a chain is called a \emph{central chain of ideals} if we have
\begin{equation}\label{cent-cond}
\mA\circ A_\ii\subseteq A_{\ii+1} \ \ \   (\ii=1,2\wdots m) .
\end{equation}
We shall soon see that Novikov algebras,  bicommutative algebras and assosymmetric algebras which possess central chains  of ideals have special properties; we investigate some of them by considering a particular central chain:
\begin{defi}\label{defi-fc}
For every Lie-admissible algebra~$\mA$ we form a series of ideals
\begin{equation}\label{def-hi}
H_1:=\mA, \  H_{\ii+1}:=H_\ii\circ \mA \mbox{ for } \ii\geq 1.
\end{equation}
We say that~$\mA$ is of \emph{finite class} if~$H_n=(0)$ for some positive integer~$n$. For the minimal integer~$n$ such that~$H_n=(0)$, we call~$n-1$ the \emph{class} of~$\mA$, and call
\begin{equation}\label{lower-cent}
\mA=H_1\supseteq H_2\supseteq \cdots \supseteq H_{n-1}\supseteq H_n=(0)
\end{equation}
the \emph{lower central chain} of~$\mA$.  To avoid too many repetitions, we shall fix the notation of~$H_\ii$ for all~$\ii\geq 1$.
\end{defi}

With the notations of~\eqref{cent-cond} and~\eqref{lower-cent}, it is straightforward to show that~$H_i\subseteq A_i$ by induction on~$i$.

The following lemma provides us with a description of commutator ideals of Novikov algebras or bicommutative algebras.

\begin{lem}\label{com-id}
Let~$\mA$ be either a Novikov algebra or a bicommutative algebra, and let~$A$ and $B$ be two Lie ideals of~$\mA$. Then we have~$A\circ B=[A,B] +\mA[A,B]$.
\end{lem}
\begin{proof} Clearly, it suffices to show that~$[A,B] +\mA[A,B]$ is an ideal of~$\mA$. In either case, $\mA$ is a Lie-admissible algebra. So we have
$$[[\xx,\yy],\zz]+[[\yy,\zz],\xx]+[[\zz,\xx],\yy]=0$$  for all~$\xx,\yy,\zz\in\mA$.	 In particular, for all~$\aa\in A$,  $\bb\in B$ and~$\xx\in \mA$, we have
$$[\aa,\bb]\xx-\xx[\aa,\bb]
=[[\aa,\bb],\xx]=[[\aa,\xx],\bb]+[\aa,[\bb,\xx]].$$
Since~$[\aa,\xx]\in A$ and~$[\bb,\xx]\in B$, we obtain  that
$$[\aa,\bb]\xx=\xx[\aa,\bb]+[[\aa,\xx],\bb]+[\aa,[\bb,\xx]]\in [A,B] +\mA[A,B].$$
For all~$\yy\in \mA$, by the right commutativity, we obtain that
$$(\xx[\aa,\bb])\yy=(\xx\yy)[\aa,\bb]\in [A,B] +\mA[A,B].$$
Moreover, since~$\mA$ is Lie-admissible algebra, we have
$$\yy(\xx[\aa,\bb])
=\yy(-[[\aa,\xx],\bb]-[\aa,[\bb,\xx]]+[\aa,\bb]\xx)
=-\yy[[\aa,\xx],\bb]-\yy[\aa,[\bb,\xx]]+\yy([\aa,\bb]\xx).$$

If~$\mA$ is a Novikov algebra, then by left symmetry and by the above reasoning, we deduce
\begin{align*}
\yy(\xx[\aa,\bb])
=&-\yy[[\aa,\xx],\bb]-\yy[\aa,[\bb,\xx]]+\yy([\aa,\bb]\xx)&\\
=&-\yy[[\aa,\xx],\bb]-\yy[\aa,[\bb,\xx]]
+(\yy[\aa,\bb])\xx+[\aa,\bb](\yy\xx)-([\aa,\bb]\yy)\xx&\\
=&-\yy[[\aa,\xx],\bb]-\yy[\aa,[\bb,\xx]]
+(\yy\xx)[\aa,\bb]+[\aa,\bb](\yy\xx)-([\aa,\bb]\yy)\xx&\\
\in &\, [A,B] +\mA[A,B].&
\end{align*}
 It follows that $[A,B] +\mA[A,B]$ is an ideal, and so we have~$A\circ B=[A,B] +\mA[A,B]$ if~$\mA$ is a Novikov algebra.

 If~$\mA$ is a bicommutative algebra, then by left commutativity and by the above reasoning, we have
\begin{align*}
\yy(\xx[\aa,\bb])
&=-\yy[[\aa,\xx],\bb]-\yy[\aa,[\bb,\xx]]+\yy([\aa,\bb]\xx)&\\
&=-\yy[[\aa,\xx],\bb]-\yy[\aa,[\bb,\xx]]+[\aa,\bb](\yy\xx)&\\
&\in \, [A,B] +\mA[A,B].&
\end{align*}
It follows that $[A,B] +\mA[A,B]$ is an ideal, and so we have~$A\circ B=[A,B] +\mA[A,B]$ if~$\mA$ is a bicommutative algebra, thus finishing the proof.
\end{proof}

It is well known (see, e.g., \cite[Theorem 2.1]{jen1}) that  an associative algebra is of finite class if and only if it has a central chain of ideals,  the proof of which is quite easy and holds for every Lie-admissible algebra.  In light of this note, we shall focus on the study of properties of the lower central chain of Novikov algebras and bicommutative algebras that are of finite class.

For associative algebras that are of finite class, Jennings \cite{jen1} showed nice connections between the central  chains~\eqref{def-cent} and~\eqref{lower-cent}. In particular, in \cite[Theorems 3.3 and 3.4]{jen1} Jennings proved that $H_p A_q\subseteq A_{\pp+\qq-1}$ and $H_p\circ A_q\subseteq A_{\pp+\qq}$ for all integers~$\pp,\qq\geq1$, and the proof was based essentially on the associativity. In the following theorem (Theorem~\ref{th-pro}), we provide an analogue of this result for  Novikov algebras and bicommutative algebras with new techniques.

\begin{lem}\label{circ-pro} Let~$\mA$ be either a Novikov algebra or a bicommutative algebra, and let~$\mA=A_1\supseteq A_2 \supseteq \cdots \supseteq A_\mm
\supseteq \cdots $ be a chain of ideals of~$\mA$  satisfying~$\mA \circ A_\ii\subseteq A_{\ii+1}$ for all~$\ii\geq 1$.
Define a series of ideals inductively by the rule
$H_1=\mA$, $H_{\ii+1}=\mA \circ H_\ii \mbox{ for } \ii\geq 1$.
Then  we have~$H_\pp\circ A_q\subseteq A_{\pp+\qq}$. In particular, we have~$H_\pp\circ H_\qq\subseteq H_{\pp+\qq}$.
\end{lem}
\begin{proof}   Obviously, for all~$\pp,\qq\geq1$, all~$H_\pp$ and~$A_\qq$ are Lie ideals of~$\mA$.
We use induction on~$\pp$ to prove the claim. For~$\pp=1$,  the claim immediately follows from the definitions of~$H_p$ and~$A_q$.  Assume that~$\pp>1$.   By Lemma~\ref{com-id}, we have
$$H_\pp \circ A_\qq=[H_\pp,A_\qq] +\mA[H_\pp,A_\qq]$$
and
$$H_p=[H_{p-1}, \mA] + \mA [H_{p-1}, \mA].$$
Since~$A_{\pp+\qq}$ is an ideal of~$\mA$, it suffices to show~$[\hh, \aa]\in A_{\pp+\qq}$  for all~$\hh\in H_\pp$, $\aa\in A_\qq$.
We need to consider the following cases.
Case 1: If~$\hh=[\hh_{\pp-1},\xx]$ for some elements~$\hh_{\pp-1}\in H_{\pp-1}$ and~$\xx\in \mA$, then by induction hypothesis, we have
\begin{multline*}
[\hh,\aa]
= [[\hh_{\pp-1},\xx],\aa]
= [[\hh_{\pp-1},\aa],\xx]+ [\hh_{\pp-1},[\xx,\aa]] \\
\in  [A_{\pp+\qq-1}, \mA]+ [H_{\pp-1}, A_{\qq+1}]
\subseteq  A_{\pp+\qq-1}\circ \mA+ H_{\pp-1}\circ A_{\qq+1}
\subseteq  A_{\pp+\qq}.
\end{multline*}
Case 2: If~$\hh=\yy[\hh_{\pp-1},\xx]$ for some elements~$\hh_{\pp-1}\in H_{\pp-1}$ and~$\xx,\yy\in \mA$, then we denote~$[\hh_{\pp-1},\xx]$ by~$\hh'$. Thus we have $\hh'\in H_\pp$ and~$\hh=\yy\hh'$.   By Case 1, for all~$\aa\in A_\qq$, we have~$[\hh',\aa]\in A_{\pp+\qq}$. In particular, we have
$[\hh',\aa]\yy,\ \yy[\hh',\aa],\ [\hh',\yy\aa]\in A_{\pp+\qq}$.

If~$\mA$ is a Novikov algebra, then by induction hypothesis, we obtain
\begin{align*}
[\hh,\aa]
=&(\yy\hh')\aa-\aa(\yy\hh')&\\
=&(\yy\aa)\hh'-(\aa\yy)\hh'
-\yy(\aa\hh')+(\yy\aa)\hh'&\\
=&(\yy\aa)\hh'-\hh'(\aa\yy)+[\hh',\aa\yy]
-\yy(\hh'\aa)+\yy[\hh',\aa]
+\hh'(\yy\aa)-[\hh',\yy\aa]&\\
=&(\yy\aa)\hh'-\hh'(\aa\yy)
-\yy(\hh'\aa)+(\hh'\yy)\aa+\yy(\hh'\aa)
-(\yy\hh')\aa
+[\hh',\aa\yy]+\yy[\hh',\aa]-[\hh',\yy\aa]&\\
=&-\hh'(\aa\yy)+(\hh'\yy)\aa
+[\hh',\aa\yy]+\yy[\hh',\aa]-[\hh',\yy\aa]\  (\mbox{by right commutativity})&\\
=&-\hh'(\aa\yy)+(\hh'\aa)\yy
+[\hh',\aa\yy]+\yy[\hh',\aa]-[\hh',\yy\aa]&\\
=&-\hh'(\aa\yy)+(\aa\hh')\yy+[\hh',\aa]\yy
+[\hh',\aa\yy]+\yy[\hh',\aa]-[\hh',\yy\aa]&\\
=&-\hh'(\aa\yy)+(\aa\yy)\hh'
+[\hh',\aa]\yy+[\hh',\aa\yy]+\yy[\hh',\aa]
-[\hh',\yy\aa]&\\
=& [\hh',\aa]\yy+\yy[\hh',\aa] -[\hh',\yy\aa]\in A_{\pp+\qq}.&
\end{align*}
Therefore, we have~$H_\pp\circ A_\qq\subseteq A_{\pp+\qq}$, thus finishing the proof for the case when~$\mA$ is a Novikov algebra.

If~$\mA$ is a bicommutative algebra, then by induction hypothesis and by the above reasoning, we obtain
\begin{align*}
[\hh,\aa]
=&(\yy\hh')\aa-\aa(\yy\hh')
=(\yy\hh')\aa-\hh'(ya)+\hh'(ya)-\aa(\yy\hh')&\\
=&(\yy\aa)\hh'-\hh'(ya)+y(\hh'a)-\yy(\aa\hh')
=[\yy\aa,\hh']+\yy [\hh',\aa]\in A_{\pp+\qq}.&
\end{align*}
Therefore, we have~$H_\pp\circ A_\qq\subseteq A_{\pp+\qq}$, thus finishing the proof.
\end{proof}

Before going further, we shall show that both Novikov algebras and bicommutative algebras are two-sided Alia~\cite{Dz09}. Note that both (left) Novikov algebras and bicommutative algebras are right commutative and Lie-admissible, the following lemma can be viewed as a corollary of~\cite[Proposition 6.1]{Dz09}.
\begin{lem}\emph{\cite[Proposition 6.1]{Dz09}}\label{lem-id}
Let~$\mA$ be a right commutative Lie-admissible algebra.  Then~$\mA$ is  a two-sided Alia algebra, more precisely, we have~$[\xx,\yy]\zz+[\yy,\zz]\xx+[\zz,\xx]\yy=0$ and~$\xx[\yy,\zz]+\yy[\zz,\xx]+\zz[\xx,\yy]=0$ for all~$\xx,\yy,\zz\in \mA$.
In particular, Novikov algebras and bicommutative algebras are  two-sided Alia.
\end{lem}
\begin{proof}
  By right commutativity, we have
  \begin{align*}
    [\xx,\yy]\zz+[\yy,\zz]\xx+[\zz,\xx]\yy
    =&(\xx\yy)\zz-(\yy\xx)\zz
    +(\yy\zz)\xx-(\zz\yy)\xx
    +(\zz\xx)\yy-(\xx\zz)\yy&\\
     =&(\xx\yy)\zz-(\xx\zz)\yy
     +(\yy\zz)\xx-(\yy\xx)\zz
    +(\zz\xx)\yy -(\zz\yy)\xx=0.&
  \end{align*}
 Moreover, since~$\mA$ is Lie-admissible, we have
 $$[\xx,\yy]\zz-\zz[\xx,\yy]+[\yy,\zz]\xx-\xx[\yy,\zz]
 +[\zz,\xx]\yy-\yy[\zz,\xx]=[[\xx,\yy],\zz]+[[\yy,\zz],\xx]+[[\zz,\xx],\yy]=0.$$
 Therefore, we obtain that
 $$\xx[\yy,\zz]+\yy[\zz,\xx]+\zz[\xx,\yy]=[\xx,\yy]\zz+[\yy,\zz]\xx
 +[\zz,\xx]\yy=0,$$
 thus finishing the proof.
\end{proof}

Recall that for an arbitrary subspace~$V$ of a Lie-admissible algebra $\mA$, we define~$V^1=V$ and~$V^{\nn}=\sum_{1\leq\ii\leq\nn-1}V^{\ii} V^{\nn-\ii}$ for all $\nn\geq2$. Then~$V$ is called \emph{nilpotent} if~$V^{\nn}=(0)$ for some positive integer~$\nn$.

We are now in position to provide the main result of this subsection, which shows that if~$\mA$ is a Novikov algebra (or bicommutative algebra) that is of finite class, then the commutator ideal~$\mA\circ \mA$ is nilpotent.

\begin{thm}\label{th-pro}
Retain the hypotheses and notations as Lemma~\ref{circ-pro}, we have~$H_\pp H_\qq\subseteq H_{\pp+\qq-1}$ for all positive integers~$\pp$ and~$\qq$. Moreover, if~$\mA$ is of finite class, then $\mA\circ \mA$ is nilpotent of nilpotent index less or equal to the class of~$\mA$.
\end{thm}
\begin{proof}
We first use induction on~$\min\{\pp,\qq\}$ to prove~$H_\pp H_\qq\subseteq H_{\pp+\qq-1}$. If~$\min\{\pp,\qq\}=1$, then it is clear because~$H_\pp$ and~$H_\qq$ are ideals of~$\mA$.
 Assume that~$\pp,\qq\geq 2$. By Lemma~\ref{circ-pro}, for all $\hh_\pp\in H_\pp$ and $\hh_\qq\in H_\qq$, we have
$$\hh_\pp\hh_\qq-\hh_\qq\hh_\pp=[\hh_\pp, \hh_\qq]
\subseteq H_\pp\circ H_\qq\subseteq H_{\pp+\qq}\subseteq H_{\pp+\qq-1}.$$	In other words, $\hh_\pp\hh_\qq$ lies in~$H_{\pp+\qq-1}$ if and only if so does~$\hh_\qq\hh_\pp$. Without loss of generality, we may assume~$\pp\leq \qq$.
If~$\hh_\pp=\xx_2[\hh_{\pp-1},\xx_1]$ for some~$\xx_1,\xx_2\in\mA$ and~$\hh_{\pp-1}\in H_{\pp-1}$, then we have
$$\hh_\pp\hh_\qq
=(\xx_2[\hh_{\pp-1},\xx_1])\hh_\qq
=(\xx_2\hh_\qq)[\hh_{\pp-1},\xx_1]
=\hh_\qq'[\hh_{\pp-1},\xx_1]$$
for~$\hh_\qq'=\xx_2\hh_\qq\in H_\qq$. So~$\hh_\pp\hh_\qq$ lies in~$H_{\pp+\qq-1}$ if and only if so does~$\hh_\qq'[\hh_{\pp-1},\xx_1]$. By the above reasoning, it suffices to show~$[\hh_{\pp-1},\xx_1]\hh_\qq'\in H_{\pp+\qq-1}$.
So we may assume~$\hh_\pp=[\hh_{\pp-1},\xx]$ for some~$\xx\in\mA$ and~$\hh_{\pp-1}\in H_{\pp-1}$. By  Lemmas~\ref{circ-pro} and~\ref{lem-id} and  by induction hypothesis, we deduce
$$
\hh_\pp\hh_\qq=[\hh_{\pp-1},\xx]\hh_\qq
=-[\xx,\hh_\qq]\hh_{\pp-1} -[\hh_\qq,\hh_{\pp-1}]\xx
\in  H_{\qq+1} H_{\pp-1}+H_{\pp-1+\qq}\xx
\subseteq  H_{\pp+\qq-1}.
$$
For the second claim, since~$\mA\circ \mA=H_2$, we shall use induction on~$\mm$ to show~$(H_2)^m\subseteq H_{m+1}$. For~$\mm=1$, there is nothing to prove. Assume that~$\mm\geq 2$. Then by induction hypothesis, we have
\begin{align*}
   (H_2)^\mm&=\sum_{1\leq\ii\leq\mm-1}(H_2)^{\ii} (H_2)^{\mm-\ii}
\subseteq\sum_{1\leq\ii\leq\mm-1}H_{\ii+1} H_{\mm-\ii+1}&\\
&\subseteq\sum_{1\leq\ii\leq\mm-1}H_{\ii+1+\mm-\ii+1-1}
=H_{\mm+1}.&
\end{align*}
Since~$\mA$ is of finite class,  we may assume $H_{\mm+1}=(0)$ for some integer $\mm$, and so $H_2$ is nilpotent of index not greater than~$\mm$, thus finishing the proof.
\end{proof}

\section{Lie nilpotent Novikov  and bicommutative algebras}\label{sec-Lie nilpotent-nov} In this section, based on Theorem \ref{th-pro}, we show that a Novikov algebra or bicommutative algebra $\mA$ is Lie nilpotent if and only if~$\mA$ is of finite class (Theorem \ref{nil-finclass}). We also find a property of Novikov algebras or bicommutative algebras on the products of commutator ideals that does not hold in general for associative algebras (Theorem \ref{prod-com-id}).

We begin this section by recalling useful notions. Let~$\mA$ be a Lie-admissible algebra. We define
$$\mA_{[1]}=\mA \ \mbox{ and }\  \mA_{[\ii+1]}=[\mA,\mA_{[\ii]}] \text{ for all } \ii\geq 1.$$
We call~$\Id{\mA_{[\ii]}}$ the \emph{$i$th commutator ideal} of~$\mA$. And the algebra~$\mA$ is called \emph{Lie nilpotent} if~$\mA_{[\ii]}=(0)$ for some integer~$\ii$.

For every integer~$\ii\geq1$, the linear space~$\mA_{[\ii]}$ is obviously a Lie ideal of~$\mA$. Therefore, by Lemma~\ref{com-id}, we immediately obtain the construction of~$\Id{\mA_{[\ii]}}$ when~$\mA$ is a Novikov algebra or bicommutative algebra.
\begin{lem}\label{lem-mni}
Let~$\mA$ be either a Novikov algebra or a bicommutative algebra. Then we have~$\Id{\mA_{[\ii]}}=\mA_{[\ii]}+\mA\mA_{[\ii]}$ for all integer~$\ii\geq 1$.
\end{lem}
Lots of interests have been attracted by the subject of commutator ideals of associative algebras.  Etingof, Kim and Ma~\cite{universal-Lie nilpotent} studied the quotient of a free algebra by its $i$-th commutator ideal, and studied the products of such commutator ideals. Kerchev~\cite{filtration} studied the filtration of a free algebra by its associative lower
central series.  We refer to~\cite{BJ,pro-com} and the references therein for a detailed history and overview of this direction.  It is worth mentioning that there is also a series of interesting results on the solvability and nilpotency of Poisson algebras with respect to their Lie structures~\cite{poi1,poi2,poi3}. In light of these notes, it is natural and interesting to study the commutator ideals of Lie-admissible algebras. Let us begin with the following easy observation.

\begin{lem}\label{lem-mni1}
Let~$\mA$ be a right commutative Lie-admissible algebra. Then we have $\mA_{[2]}\mA_{[\ii]}\subseteq \Id{\mA_{[\ii+1]}}$ for all integer~$\ii\geq 1$. In particular, for every Novikov algebra or bicommutative algebra~$\mA$, we have $\mA_{[2]}\mA_{[\ii]}\subseteq \Id{\mA_{[\ii+1]}}$ for all integer~$\ii\geq 1$.
\end{lem}
\begin{proof}
For all~$\aa,\bb\in\mA$, for all~$\cc_\ii\in\mA_{[\ii]}$, by Lemma~\ref{lem-id}, we have
$$
[\aa,\bb]\cc_\ii=-[\bb,\cc_\ii]\aa-[\cc_\ii,\aa]\bb
 \in \Id{\mA_{[\ii+1]}},$$
 thus finishing the proof.
\end{proof}

Recall that for an arbitrary Lie-admissible algebra, we have~$H_1=\mA$ and~$H_{\ii+1}=\mA\circ H_\ii$ for all~$\ii\geq 1$.  In the following theorem we show that~$\Id{\mA_{[\ii]}}= H_\ii$ if~$\mA$ is either a Novikov algebra or a bicommutative algebra.  In particular, Theorem~\ref{th-pro} then describes the products of the commutator ideals of~$\mA$, which is in general not true for Lie nilpotent associative algebras.

\begin{thm}\label{prod-com-id}
Let~$\mA$ be either a Novikov algebra or a bicommutative algebra. Define a series of ideals by the rule
$H_1=\mA$, $H_{\ii+1}=\mA \circ H_\ii \mbox{ for all } \ii\geq 1$. Then we have~$\Id{\mA_{[\ii]}}=H_\ii$ for all integer~$\ii\geq 1$. In particular, we have~$\Id{\mA_{[\pp]}}\Id{\mA_{[\qq]}}\subseteq \Id{\mA_{[\pp+\qq-1]}}$ for all integers~$\pp,\qq\geq 1$.
\end{thm}
\begin{proof}
By induction on~$\ii$, it is straightforward to show~$\Id{\mA_{[\ii]}}\subseteq H_\ii$. So it suffices to prove that~$H_\ii\subseteq\Id{\mA_{[\ii]}}$ for all integer~$\ii\geq 1$, and we shall use induction on~$\ii$ to prove this claim. For~$\ii\leq 2$, it is clear. Assume that~$\ii\geq 3$. By induction hypothesis and  by Lemma~\ref{lem-mni}, $H_\ii$ is generated by
$$[\mA, H_{\ii-1}]=[\mA, \mA_{[\ii-1]}]+[\mA, \mA\mA_{[\ii-1]}]
=\mA_{[\ii]}+[\mA, \mA\mA_{[\ii-1]}].$$
Therefore, it suffices to show~$[\mA, \mA\mA_{[\ii-1]}]\subseteq \Id{\mA_{[\ii]}}$.

  If~$\mA$ is a Novikov algebra, then for all~$\aa,\bb,\cc\in \mA$ and~$\dd_{\ii-2}\in\mA_{[\ii-2]}$, by Lemma \ref{lem-mni1}, we  deduce
\begin{align*}
[\aa, \bb[\cc,\dd_{\ii-2}]]
=&\aa(\bb[\cc,\dd_{\ii-2}])-(\bb[\cc,\dd_{\ii-2}])\aa&\\
=&\aa([\bb,[\cc,\dd_{\ii-2}]]
+[\cc,\dd_{\ii-2}]\bb)-(\bb\aa)[\cc,\dd_{\ii-2}]&\\
=&\aa[\bb,[\cc,\dd_{\ii-2}]]
+\aa([\cc,\dd_{\ii-2}]\bb)-(\bb\aa)[\cc,\dd_{\ii-2}]&\\
=&\aa[\bb,[\cc,\dd_{\ii-2}]]
+(\aa[\cc,\dd_{\ii-2}])\bb+[\cc,\dd_{\ii-2}](\aa\bb)
-([\cc,\dd_{\ii-2}]\aa)\bb-(\bb\aa)[\cc,\dd_{\ii-2}]&\\
=&\aa[\bb,[\cc,\dd_{\ii-2}]]
+(\aa\bb)[\cc,\dd_{\ii-2}]-(\bb\aa)[\cc,\dd_{\ii-2}]
+[\cc,\dd_{\ii-2}](\aa\bb)-([\cc,\dd_{\ii-2}]\aa)b&\\
=&\aa[\bb,[\cc,\dd_{\ii-2}]]
+[\aa,\bb][\cc,\dd_{\ii-2}]
+[\cc,\dd_{\ii-2}](\aa\bb)-(\aa[\cc,\dd_{\ii-2}])b
+[\aa,[\cc,\dd_{\ii-2}]]b&\\
=&\aa[\bb,[\cc,\dd_{\ii-2}]]
+[\aa,\bb][\cc,\dd_{\ii-2}]
+[\cc,\dd_{\ii-2}](\aa\bb)-(\aa\bb)[\cc,\dd_{\ii-2}]
+[\aa,[\cc,\dd_{\ii-2}]]b&\\
=&\aa[\bb,[\cc,\dd_{\ii-2}]]
+[\aa,\bb][\cc,\dd_{\ii-2}]
-[\aa\bb, [\cc,\dd_{\ii-2}]]
+[\aa,[\cc,\dd_{\ii-2}]]b&\\
\in &\Id{\mA_{[\ii]}} +\mA_{[2]}\mA_{[\ii-1]}\subseteq \Id{\mA_{[\ii]}}.&
\end{align*}

 If~$\mA$ is a bicommutative algebra, then for all~$\aa,\bb\in \mA$ and~$\dd_{\ii-1}\in\mA_{[\ii-1]}$, we  deduce
\begin{align*}
[\aa, \bb\dd_{\ii-1}]
=& \aa(\bb\dd_{\ii-1})-(\bb\dd_{\ii-1})\aa&\\
=& \bb(\aa\dd_{\ii-1})-(\bb\aa)\dd_{\ii-1}&\\
 =&\bb(\aa\dd_{\ii-1})-\bb(\dd_{\ii-1}\aa)
 +\bb(\dd_{\ii-1}\aa)-(\bb\aa)\dd_{\ii-1}&\\
 =&\bb[\aa,\dd_{\ii-1}]
 +\dd_{\ii-1}(\bb\aa)-(\bb\aa)\dd_{\ii-1}&\\
 =&\bb[\aa,\dd_{\ii-1}]
 -[\bb\aa,\dd_{\ii-1}]&\\
\in & \Id{\mA_{[\ii]}}.&
\end{align*}
Therefore, we have~$\Id{\mA_{[\ii]}}=H_\ii$ for all integer~$\ii\geq 1$.

Finally,  by Theorem~\ref{th-pro}, we have
$$\Id{\mA_{[\pp]}}\Id{\mA_{[\qq]}}=H_pH_q\subseteq H_{p+q-1}= \Id{\mA_{[\pp+\qq-1]}},$$ thus finishing the proof.
\end{proof}

Consequently, we obtain the main theorem of this section, which shows that the notions of ``finite class" and ``Lie nilpotency" for Novikov algebras and bicommutative algebras are the same.

\begin{thm}\label{nil-finclass}
Let~$\mA$ be either a Novikov algebra or a bicommutative algebra. Then~$\mA$ is Lie nilpotent if and only if~$\mA$ is of finite class. Consequently, the ideal of~$\mA$ generated by~$\{ab-ba\mid \aa,\bb\in\mA\}$ is nilpotent  if~$\mA$ is Lie nilpotent. 
\end{thm}
\begin{proof} By Theorem~\ref{prod-com-id}, we have~$\Id{\mA_{[\ii]}}=H_\ii$ for all integer~$\ii\geq 1$. It follows that $\mA$ is Lie nilpotent if and only if $\mA$ is of finite class. By Theorem~\ref{th-pro}, the second claim follows, thus finishing the proof.
\end{proof}

Theorems~\ref{prod-com-id} and~\ref{nil-finclass} indicates that bicommutative algebras and Novikov algebras possess lots of close properties.  We also note that  over a field of characteristic $\neq 2$, a Novikov algebra  $\mA$ is Lie-solvable if and only if the ideal of~$\mA$ generated by~$\{ab-ba\mid \aa,\bb\in\mA\}$ is right nilpotent~\cite{TUZ}. But an analog does not hold for Lie-solvable bicommutative algebras in general. It is well-known that for every bicommutative algebra~$\mB$,
 the subalgebra~$\mB^2=\{ab\mid a,b\in \mB\}$ is commutative and associative~\cite{DIT11}. So~$\mB$ is Lie-metabelian, namely, for all~$a,b,c,d\in \mB$, we have~$[[a,b],[c,d]]=0$. However, it is easy to see that for the free bicommutative algebra~$\mB$ generated by more than two elements,  the ideal of~$\mB$ generated by~$\{ab-ba\mid \aa,\bb\in\mB\}$ is not right/left nilpotent.

  We conclude this section with an example constructed by  S. Pchelintsev, which shows that some conditions are missing in~\cite[Theorem 6.6]{jen1}.  More precisely,
it is claimed~\cite[Theorem 6.6]{jen1} that an associative algebra is Lie nilpotent if and only if it is of finite class, but the following example shows that this theorem does not hold in general.

 \begin{exam}\label{ex-gra}(S. Pchelintsev)
   Let~$\mA$ be the Grassmann algebra
 generated by the infinite set~$\{e_i \mid i\in \mathbb{N}\}$ over a field of characteristic not 2. Then~$\mA$ is Lie nilpotent of index 3, namely, $[[\mA,\mA], \mA]=0$, but~$ \mA$ is not of finite class and $\mA \circ \mA$ is not nilpotent.
 \end{exam}

\begin{proof}
  Note that the set $X:=\{e_{i_1}\dots e_{i_n}\mid i_1<\cdots < i_n, n \geq 1\}$ forms a linear basis of~$\mA$ (if we assume that $\mA$ is non-unital). Moreover,  elements in~$X$ of the form~$e_{i_1}\dots e_{i_{2n}}$ lies in the center of~$\mA$. It follows immediately that~$[[\mA,\mA], \mA]=0$.

  We claim that~$e_1e_2...e_{2n}$ lies in~$H_{n+1}$, where $H_{n+1}$ is defined in Definition~\ref{defi-fc}. For~$n=1$, we have~$[e_1,e_2]=e_1e_2-e_2e_1=2e_1e_2\in H_2$. Since the characteristic of the underlying field is not 2, it follows that~$e_1e_2\in H_2$. By induction hypothesis, $e_1e_2...e_{2n}\in H_{n+1}$ and thus~$e_1e_2...e_{2n}e_{2n+1}\in H_{n+1}$. It follows that $[e_1e_2...e_{2n}e_{2n+1},e_{2n+2}]\in H_{n+2}$ and thus~$e_1e_2...e_{2n}e_{2n+1}e_{2n+2}\in H_{n+2}$. Finally, note that~$e_{2i-1}e_{2i}=(1/2)[e_{2i-1},e_{2i}]\in \mA\circ \mA$ and~$e_1e_2...e_{2n}\neq 0$. The result follows.
\end{proof}

\section{Commutator ideals of assosymmetric algebras}\label{sec-4}
In this section we mainly study assosymmetric algebras of finite class and study commutator ideals of assosymmetric algebras. We show that some of the properties for associative algebras also hold for assosymmetric algebras, namely, for such properties the associativity is not necessary and can be replaced by left symmetry and right symmetry.

\subsection{Assosymmetric algebras of finite class} The aim of this subsection is to study assosymmetric algbras of finite class. We shall show that  assosymmetric algbras of finite class have properties similar to that of associative algebras.  Since the associativity does not hold for assosymmetric algebras, new techniques are needed. Recall that for all~$x,y,z$ in an algebra~$\mA$, the associator $(x,y,z)$ means~$(xy)z-x(yz)$. So in every assosymmetric algebra~$\mA$, we have~$(y,x,z)=(x,y,z)=(x,z,y)$ for all~$x,y,z\in \mA$.

From now on, $\mA$ always means an assosymmetric algebra over a field~$F$.
It is proved in~\cite{Kl57} that for all~$x,y,z,w$ in an assosymmetric algebra~$\mA$,  we have\begin{equation}\label{eq3.1.4}
([x,y],z,w)=0 \mbox{  if } \mathsf{char}(F)\neq 2, 3.
\end{equation}
 By the same technique developed in~\cite{Kl57}, we obtain some more identities as follows when~$\mathsf{char}(F)=2$ or $3$.
\begin{lem}
For all~$x,y,z,w$ in an assosymmetric algebra~$\mA$, we have
\begin{equation}\label{eq3.1.1}
  (x,y,z)=-[x,y]z+x[y,z]+[xz,y];
\end{equation}
\begin{equation}\label{eq3.1.2}
([w,x],y,z)=[w,(x,y,z)]+[x,(w,y,z)] \mbox{  if } \mathsf{char}(F)=2;
\end{equation}
\begin{equation}\label{eq3.1.3}
[xy,z]=-[yz,x]-[zx,y] \mbox{  if } \mathsf{char}(F)=3;
\end{equation}
\end{lem}
\begin{proof}
\ITEM1 Proof of identity~\eqref{eq3.1.1}. Note that
\begin{align*}
  (x,y,z)&=-(x,y,z)+(y,x,z)+(x,z,y)&\\
  &=-(xy)z+x(yz)+(yx)z-y(xz)+(xz)y-x(zy)&\\
    &=-(xy)z+(yx)z+x(yz)-x(zy)+(xz)y-y(xz)&\\
    &=-[x,y]z+x[y,z]+[xz,y].
\end{align*}
  The proof of identity~\eqref{eq3.1.1} is completed.

\ITEM2 Proof of identity~\eqref{eq3.1.2}.  Following~\cite{Kl57}, we define
$$f(w,x,y,z)=(wx,y,z)-x(w,y,z)-(x,y,z)w.$$
Then it is obvious that
\begin{equation}\label{eq-fyz}
   f(w,x,y,z)=f(w,x,z,y).
\end{equation}
We also note that~\cite{Kl57} in any algebra we have
\begin{equation}\label{eq-any-ring}
 (wx,y,z)-(w,xy,z)+(w,x,yz)=w(x,y,z)+(w,x,y)z.
 \end{equation}
 By identity~\eqref{eq-any-ring}, we deduce
\begin{align*}
  &f(w,x,y,z)+f(z,w,x,y)&\\
  =&(wx,y,z)-x(w,y,z)-(x,y,z)w+(zw,x,y)-w(z,x,y)-(w,x,y)z&\\
  =&(wx,y,z)-x(w,y,z)-(x,y,z)w+(zw,x,y)-(wx,y,z)+(w,xy,z)-(w,x,yz)&\\
  =&-x(w,y,z)-(x,y,z)w+(zw,x,y)+(w,xy,z)-(w,x,yz)&\\
  =&-x(w,y,z)-(x,y,z)w+(xy,z,w)-(x,yz,w)+(x,y,zw)=0.&
\end{align*}
Combining this with identity~\eqref{eq-fyz}, we obtain
\begin{equation}\label{eq4.7}
f(w,x,y,z)=-f(z,w,x,y)=f(y,z,w,x),
\end{equation}
and thus
\begin{equation}\label{eq4.8}
f(w,x,y,z)=f(y,z,w,x)=f(y,z,x,w)=f(x,w,y,z).
\end{equation}
Therefore, if~$\mathsf{char}(F)=2$, we obtain
\begin{align*}
  0&=2f(w,x,y,z)&\\
  &=f(w,x,y,z)+f(x,w,y,z)&\\
  &=(wx,y,z)+x(w,y,z)+(x,y,z)w+(xw,y,z)+w(x,y,z)+(w,y,z)x&\\
   &=([w,x],y,z)+[w,(x,y,z)]+[x, (w,y,z)].&
  \end{align*}
  The proof of identity~\eqref{eq3.1.2} is completed.

\ITEM3 Proof of identity~\eqref{eq3.1.3}. If~$\mathsf{char}(F)=3$, then we have
\begin{align*}
 [xy,z]+[yz,x]+[zx,y]
&=(xy)z-z(xy)+(yz)x-x(yz)+(zx)y-y(zx)  &\\
&=(x,y,z)+(y,z,x)+(z,x,y)=3(x,y,z)=0.  &
\end{align*}
Identity~\eqref{eq3.1.3} follows immediately.
\end{proof}

 Now we begin to study associators involving Lie ideals of an assosymmetric algebra~$\mA$.
 \begin{lem}\label{lem-ideal}
Let~$A$ and~$B$ be Lie ideals of~$\mA$. Then the following statements are true:

\ITEM1 For all~$x\in B$, $y,z\in \mA$,
$(y,x,z)\in \mA[B,\mA]+[B,\mA]$;
In particular, $(y,x,z)$ is contained in the ideal of~$\mA$ generated by $[B, \mA]$;

\ITEM2 $A\circ B=[A,B] +\mA[A,B]=[A,B]+[A,B]\mA$.
 \end{lem}

 \begin{proof}
\ITEM1  By identity~\eqref{eq3.1.1}, we deduce
\begin{align*}
(y,x,z)&=-[y,x]z+y[x,z]+[yz,x]    =-[[y,x],z]-z[y,x]+y[x,z]+[yz,x]&\\
   &\in \mA[B,\mA]+[B,\mA].&
\end{align*}
  The proof is completed.

  \ITEM2 Clearly, $[A,B]$ is an Lie ideal of~$\mA$. It follows that
  $$[A,B]\mA\subseteq \mA[A,B]+[[A,B],\mA]\subseteq\mA[A,B]+[A,B]. $$
  In particular, $[A,B] +\mA[A,B]$ is an ideal of~$\mA$ if and only if so does~$[A,B] +[A,B]\mA$.

 By~\ITEM1, for all~$x,y\in \mA$, $a\in A$, $b\in B$, we have
 $$(x,[a,b],y)\in [[A,B],\mA]+\mA[[A,B],\mA]\subseteq \mA[A,B]+[A,B].$$
 It follows that
    $$x(y[a,b])=(xy)[a,b]-(x,y,[a,b])=(xy)[a,b]-(x,[a,b],y) \in \mA[A,B]+[A,B].$$
 Therefore, we deduce
$$ (x[a,b])y=(x,[a,b],y)+x([a,b]y)
=(x,[a,b],y)+x(y[a,b])+x[[a,b],y]\in \mA[A,B]+[A,B].$$
The proof is completed.
 \end{proof}

\begin{coro}\label{coro-asso-a}
  Let~$\{A_p\mid p\geq 1\}$ be a family of ideals of~$\mA$ such that~$\mA\circ A_p\subseteq A_{p+1}$ for every~$p\geq 1$. Then for all~$a\in A_p$, for all~$x,y\in \mA$, we have~$(x,a,y)\in A_{p+1}$.
  \end{coro}
 \begin{proof}
   By Lemma~\ref{lem-ideal}, we have
   $$(x,a,y) \in  [A_p,\mA] +\mA[A_p,\mA]=\mA\circ A_p\subseteq A_{p+1}.$$
   The proof is completed.
 \end{proof}

Let
\begin{equation}\label{def-cent}
\mA=A_1\supseteq A_2 \supseteq \cdots \supseteq A_\mm
\supseteq A_{\mm+1}=(0)
\end{equation}
be a central chain of ideals of~$\mA$.  And let~$H_i$  $(i\geq 1)$ be as in Definition~\ref{defi-fc}. When~$\mA$ is assosymmetric, we have the following analogs as those for associative algebras. Again, as the associativity does not hold, new techniques are necessary.

\begin{lem}\label{lem-ass-ap}
Let~$\mA$ be an assosymmetric algebra. Then we have $H_pA_q\subseteq A_{p+q-1}$, $A_qH_p\subseteq A_{p+q-1}$, $[H_p, A_q]\subseteq A_{p+q}$ and~$(H_p, A_q, \mA)\subseteq A_{p+q}$. In particular, we have~$H_pH_q\subseteq H_{p+q-1}$.
\end{lem}
 \begin{proof}
Since~$H_p\subseteq A_p$, $A_{p+q}\subseteq A_{p+q-1}$ and~$A_qH_p\subseteq [A_q,H_p]+H_pA_q$, it suffices to prove~$H_pA_q\subseteq A_{p+q-1}$,  $[H_p, A_q]\subseteq A_{p+q}$ and~$(H_p, A_q, \mA)\subseteq A_{p+q}$.

   We use induction on~$p$ to prove these claims.  For~$p=1$, we
   have~$H_1A_q\subseteq A_{q}$,  $[H_1, A_q]\subseteq A_{q+1}$ and by Corollary~\ref{coro-asso-a}, we obtain
 $$(H_1, A_q, \mA)\subseteq (\mA, A_q, \mA) \subseteq \mA\circ A_q\subseteq A_{q+1}.$$
 Now we assume~$p\geq 2$. If~$h_p=[h_{p-1},x]$ for some~$h_{p-1}\in H_{p-1}$ and~$x\in \mA$, then for every~$a\in A_q$, by induction hypothesis,  we have
 \begin{align*}
   h_pa&=[h_{p-1},x]a&\\
   &\overset{\eqref{eq3.1.1}}{=}
-(h_{p-1},x,a)+h_{p-1}[x,a]+[h_{p-1}a,x]&\\
 &\in
(H_{p-1},A_q,\mA)+H_{p-1}A_{q+1}+[H_{p-1}A_q,\mA]&\\
&\subseteq
A_{p+q-1}.&
 \end{align*}
 By the Jacobi identity and induction hypothesis, we obtain
 $$[h_p,a]=[[h_{p-1},x],a]=[[h_{p-1},a],x]+[h_{p-1},[x,a]]
 \in [A_{p+q-1},\mA]+[H_{p-1},A_{q+1}]
\subseteq  A_{p+q}.$$

We continue to show~$([h_{p-1},x], a, w)\in A_{p+q}$ for every~$w\in \mA$. There are several cases to discuss depending on the characteristic of the field. If~$\mathsf{char}(F)\neq 2,3$, then by identity~\eqref{eq3.1.4}, we have~$([h_{p-1},x], a, w)=0\in A_{p+q}$.
 If~$\mathsf{char}(F)=2$, then by identity~\eqref{eq3.1.2} and by Corollary~\ref{coro-asso-a}, we have
 $$([h_{p-1},x], a, w)=[h_{p-1},(x,a,w)]+[x,(h_{p-1},a,w)]
 \in [H_{p-1},A_{q+1}]+[\mA, A_{p+q-1}]\subseteq  A_{p+q}.
 $$
  If~$\mathsf{char}(F)=3$, then by identities~\eqref{eq3.1.1} and~\eqref{eq3.1.3} and by the above reasoning, we have
\begin{align*}
  &([h_{p-1},x], a, w)&\\
  =&-[[h_{p-1},x], a]w+[h_{p-1},x][a,w]+[[h_{p-1},x]w,a]&\\
  =&-[[h_{p-1},x], a]w+[h_{p-1},x][a,w]-[wa,[h_{p-1},x]]-[a[h_{p-1},x],w]&\\
  \in &A_{p+q}+[h_{p-1},x]A_{q+1}+[A_{p+q-1}, \mA]
  \subseteq A_{p+q}.&
\end{align*}

Now we prove for the case when~$p\geq 2$ and~$h_p=[h_{p-1},x]y$ for some elements~$h_{p-1}\in H_{p-1}$ and~$x,y\in\mA$. By the above reasoning and by the right-symmetric identity, we have
$$h_pa=([h_{p-1},x],y,a)+[h_{p-1},x](ya)
=([h_{p-1},x],a,y)+[h_{p-1},x](ya)\in  A_{p+q-1}.$$
By identity~\eqref{eq3.1.1} and by the above reasoning, we obtain
\begin{align*}
  [h_p,a]=[[h_{p-1},x]y,a]&=([h_{p-1},x],a,y)+[[h_{p-1},x],a]y-[h_{p-1},x][a,y]&\\
  &\in A_{p+q}+[h_{p-1},x]A_{q+1}\subseteq A_{p+q}.&
\end{align*}
Finally, by the above reasoning and by induction hypothesis, for every~$w\in \mA$, we deduce
\begin{align*}
  ([h_{p-1},x]y,a,w)
  &\overset{\eqref{eq3.1.1}}{=}
  -[[h_{p-1},x]y, a]w+([h_{p-1},x]y)[a,w]+[([h_{p-1},x]y)w,a]&\\
  &\in [H_{p},A_q]\mA+H_pA_{q+1}+[H_p, A_q]\subseteq A_{p+q}.&
\end{align*}
The proof is completed.
 \end{proof}
 By Lemma~\ref{lem-ass-ap} and by a similar reasoning as the proof for Theorem~\ref{th-pro}, we immediately obtain the following description for assosymmetric algebras that generalizes the corresponding result of associative algebras.

 \begin{thm}\label{thm45}
   Let~$\mA$ be an assosymmetric algebra of finite class. Then $\mA\circ \mA$ is nilpotent of nilpotent index less or equal to the class of~$\mA$.
 \end{thm}
 \subsection{Products of commutator ideals of assosymmetric algebras}
The aim of this subsection is to study products of commutator ideals of an arbitrary assosymmetric algebra~$\mA$ over a field~$F$ such that~$\mathsf{char}(F)\neq 2,3$. We shall prove that~$\Id{\mA_{[i]}}\Id{\mA_{[j]}}\subseteq \Id{\mA_{[i+j-1]}}$ if~$i$ is odd or $j$ is odd,  which generalizes the corresponding result~\cite[Corollary 1.4]{BJ} for associative algebras.
The main idea of this subsection comes from~\cite{BJ}. But since the associativity does not hold, we need new techniques.

For all~$x,y\in \mA$, we define
$$x\ast y=xy+yx \ \mbox{ and }\  [x,y]=xy-yx.$$
The main difference for the above mentioned result between associative algebras and assosymmetric algebras is the proof of the following lemma.
\begin{lem}\label{lem-46}
  Let~$\mA$ be an assosymmetric algebra over a field $F$  such that~$\mathsf{char}(F)\neq 2,3$. For every positive odd integer~$j$, we have~$[\Id{\mA_{[j]}},\mA]\subseteq \mA_{[j+1]}$. Moreover, we have~$[\Id{\mA_{[j]}},\mA_{[i]}]\subseteq \mA_{[i+j]}$.
\end{lem}
\begin{proof} For all~$x,y,z\in\mA$, we have~$[x,[y,z]]=[[x,y ],z]-[[x,z],y]$. So the second claim follows immediately from the first one.
We use induction on~$j$ to prove the lemma.  For~$j=1$, the claim follows immediately by the definition of~$\mA_{[2]}$ and by the above reasoning if~$i\geq 2$. Now we assume that~$j$ is an odd integer such that~$j\geq 3$. For all~$x,y,z,u,v\in \mA$, it suffices to show~$[x[y,[z,u]], v]\in \mA_{[j+1]}$ if~$u\in \mA_{[j-2]}$.
By assumption, we have~$\mathsf{char}(F)\neq 2$, so we have
$$x[y,[z,u]]=(1/2)([x,[y,[z,u]]]+x\ast[y,[z,u]]).$$
So in order to show~$[x[y,[z,u]],v]\in A_{[j+1]}$, it suffices to prove~$[x\ast[y,[z,u]],v]\in A_{[j+1]}$. The idea of the proof is to show that~$[x\ast[y,[z,u]],v]$ is sort of skew symmetric. More precisely, we shall prove that, if one of~$x,y,z,u,v$ lies in~$\mA_{[j-2]}$ then
$$[x\ast[y,[z,u]], v]\equiv [x\ast[z,[u,y]], v]\equiv [x\ast[u,[y,z]], v] \mod A_{[j+1]}.$$
Since~$\mA$ is an assosymmetric algebra, by identity~\eqref{eq3.1.4}, we have
$$(x,y,[z,u])=(x,[z,u],y)=([z,u],x,y)=0=(xy)[z,u]-x(y[z,u])=(x[z,u])y-x([z,u]y),$$ and thus
\begin{align*}
&x\ast[y,[z,u]]+y\ast[x,[z,u]]&\\
=&x(y[z,u]-[z,u]y)+(y[z,u]-[z,u]y)x
+y(x[z,u]-[z,u]x)+(x[z,u]-[z,u]x)y&\\
=&(xy)[z,u]-x([z,u]y)+y([z,u]x)-[z,u](yx)&\\
&+(yx)[z,u]-y([z,u]x)+x([z,u]y)-[z,u](xy)&\\
=&(xy+yx)[z,u]-[z,u](xy+yx)&\\
=&[x\ast y, [z,u]].&
\end{align*}

Now we assume that one of~$x,y,z,u,v$ lies in~$\mA_{[j-2]}$.
Since~$[\mA_{[i]}, \mA_{[t]}]\subseteq \mA_{[i+t]}$,
 by induction hypothesis, we obtain that
 $[[x\ast y, [z,u]],v] $ lies in~$A_{[j+1]}$,
and thus we deduce
\begin{equation}\label{chan-12}
  [x\ast[y,[z,u]],v]\equiv -[y\ast[x,[z,u]],v] \mod A_{[j+1]}.
\end{equation}
Similarly,  we have
\begin{align*}
&[x\ast[y,[z,u]],v]+[v\ast[y,[z,u]],x]&\\
=& (x[y,[z,u]]+[y,[z,u]]x)v-v(x[y,[z,u]]+[y,[z,u]]x)&\\
&+(v[y,[z,u]]+[y,[z,u]]v)x-x(v[y,[z,u]]+[y,[z,u]]v)&\\
=& x([y,[z,u]]v)+[y,[z,u]](xv)-(vx)[y,[z,u]]-v([y,[z,u]]x)&\\
&+v([y,[z,u]]x)+[y,[z,u]](vx)-(xv)[y,[z,u]]-x([y,[z,u]]v)&\\
=&  [y,[z,u]](x\ast v)-(x\ast v)[y,[z,u]] &\\
= &-[x\ast v,[y,[z,u]]].&
\end{align*}
Again,  since one of~$x,y,z,u,v$ lies in~$\mA_{[j-2]}$, by induction hypothesis,
we easily obtain that~$[x\ast v,[y,[z,u]]]$ lies in~$A_{[j+1]}$, and thus we deduce
\begin{equation}\label{chan-15}
[x\ast[y,[z,u]],v]\equiv -[v\ast[y,[z,u]],x] \mod A_{[j+1]}.
\end{equation}
On the other hand,  by the Jacobi identity and by  identity~\eqref{chan-12}, we have
\begin{align*}
  [x\ast[y,[z,u]],v]&=-[x\ast[z,[u,y]],v]-[x\ast[u,[y,z]],v]&\\
  &\equiv [z\ast[x,[u,y]],v]+[u\ast[x,[y,z]],v] \ \mod \mA_{[j+1]};&
\end{align*}
Interchanging~$x$ and~$y$ in the above equation, we obtain
$$ [y\ast[x,[z,u]],v]\equiv [z\ast[y,[u,x]],v]+[u\ast[y,[x,z]],v] \ \mod \mA_{[j+1]}. $$
So by the above two Equations and by the Jacobi identity, we deduce
\begin{align*}
  &2[x\ast[y,[z,u]],v]&\\
  \equiv &[x\ast[y,[z,u]],v]-[y\ast[x,[z,u]],v]&\\
  \equiv & [z\ast[x,[u,y]],v]+[u\ast[x,[y,z]],v]
  -[z\ast[y,[u,x]],v]-[u\ast[y,[x,z]],v]&\\
  \equiv &  [z\ast[u,[x,y]],v]-[u\ast[z,[x,y]],v]&\\
    \equiv &  2[z\ast[u,[x,y]],v] \ \mod \mA_{[j+1]}.&
\end{align*}
Since~$\mathsf{char}(F)\neq 2$, we obtain
\begin{equation}\label{chan-1234}
  [x\ast[y,[z,u]],v]\equiv [z\ast[u,[x,y]],v] \mod \mA_{[j+1]}.
\end{equation}
Therefore,   in the vector space~$\mA/\mA_{[j+1]} $, we have
\begin{equation}\label{chan-25}
  [x\ast[y,[z,u]],v]
\overset{\eqref{chan-12}}{\equiv}  -[y\ast[x,[z,u]],v]
\overset{\eqref{chan-15}}{\equiv}  [v\ast[x,[z,u]],y]
\overset{\eqref{chan-12}}{\equiv} - [x\ast[v,[z,u]],y],
\end{equation}
\begin{equation}\label{chan-35}
[x\ast[y,[z,u]],v]
\overset{\eqref{chan-1234}}{\equiv}  [z\ast[u,[x,y]],v]
\overset{\eqref{chan-15}}{\equiv}  -[v\ast[u,[x,y]],z]
\overset{\eqref{chan-1234}}{\equiv}  [x\ast[y,[v,u]],z],
\end{equation}
and thus
\begin{equation}\label{chan-45}
 [x\ast[y,[z,u]],v]=-[x\ast[y,[u,z]],v]\overset{\eqref{chan-35}}{\equiv} [x\ast[y,[v,z]],u]\equiv -[x\ast[y,[z,v]],u].
 \end{equation}
 Therefore, we deduce
 \begin{align*}
    [x\ast[y,[z,u]],v]
    &\overset{\eqref{chan-25}}{\equiv} - [x\ast[v,[z,u]],y]
    \overset{\eqref{chan-45}}{\equiv} [x\ast[v,[z,y]],u]&\\
    &\overset{\eqref{chan-25}}{\equiv} - [x\ast[u,[z,y]],v]
       \equiv [x\ast[u,[y,z]],v] \mod \mA_{[j+1]}.&
 \end{align*}
 It follows that
 $$[x\ast[y,[z,u]],v] \equiv [x\ast[u,[y,z]],v] \equiv [x\ast[z,[u,y]],v]  \mod \mA_{[j+1]}. $$
Finally, since~$\mA$ is Lie-admissible, we obtain
$$3[x\ast[y,[z,u]],v] \equiv
 [x\ast[y,[z,u]],v]+[x\ast[u,[y,z]],v] +[x\ast[z,[u,y]],v]
 \equiv 0 \mod \mA_{[j+1]}.$$
Since~$\mathsf{char}(F)\neq 3$,  we have~$[x\ast[y,[z,u]],v] \in \mA_{[j+1]}$. The proof is completed.
\end{proof}

We conclude the article with the main result of this subsection, which generalizes the corresponding property of associative algebras.
\begin{thm}\label{cp-ass}
Let~$\mA$ be an assosymmetric algebra. Then we have~$\Id{\mA_{[i]}} \Id{\mA_{[j]}}\subseteq \Id{\mA_{[i+j-1]}}$ if~$i$ or $j$ is odd.
\end{thm}
\begin{proof}
 If~$i=1$ or $j=1$, then  clearly we have~$\Id{\mA_{[i]}} \Id{\mA_{[j]}}\subseteq \Id{\mA_{[i+j-1]}}$.
Now we assume~$i\geq 2$ and~$j\geq 2$.  Then by Lemma~\ref{lem-ideal}\ITEM2 and by identity~\eqref{eq3.1.4}, we have
\begin{align*}
  \Id{\mA_{[i]}} \Id{\mA_{[j]}}
  &= (\mA_{[i]}+\mA\mA_{[i]})(\mA_{[j]}+\mA_{[j]}\mA)&\\
  &\subseteq \mA_{[i]}\mA_{[j]}+\mA(\mA_{[i]}\mA_{[j]})
+(\mA_{[i]}\mA_{[j]})\mA+\mA(\mA_{[i]}\mA_{[j]})\mA.&
\end{align*}
So it suffices to show~$\mA_{[i]}\mA_{[j]}\subseteq \Id{\mA_{[i+j-1]}}$ if one of~$i$ and~$j$ is odd.  Since
$$\mA_{[i]}\mA_{[j]}
\subseteq [\mA_{[i]}, \mA_{[j]}]+\mA_{[j]}\mA_{[i]}
\subseteq  \mA_{[i+j]}+\mA_{[j]}\mA_{[i]},$$
we may assume that~$j$ is odd and thus we may assume~$j\geq 3$ and~$i\geq 2$. For all~$x\in \mA$, $y\in \mA_{[i-1]}$ and $z\in \mA_{[j]}$, by identity~\eqref{eq3.1.4} and by Lemma~\ref{lem-46}, we have
\begin{align*}
  [x,y]z&=(xy)z-(yx)z=x(yz)-y(xz)=x(yz)-x(zy)+x(zy)-y(xz)&\\
  &=x(yz)-x(zy)+(xz)y-y(xz)=x[y,z]+[xz,y]&\\
   &\in \mA \mA_{[i+j-1]}+[\Id{\mA_{[j]}},\mA_{[i-1]}]\subseteq \Id{\mA_{[i+j-1]}}.&
\end{align*}
 The proof is completed.
\end{proof}

We also note that if~$i$ and $j$ are even then $\Id{\mA_{[i]}}\Id{\mA_{[j]}} \nsubseteqq \Id{\mA_{[i+j-1]}}$ in general for associative algebras~\cite{pro-com}. Since associative algebras are assosymmetric algebras,  we know that
if~$i$ and $j$ are even then $\Id{\mA_{[i]}}\Id{\mA_{[j]}} \nsubseteqq \Id{\mA_{[i+j-1]}}$ in general for assosymmetirc algebras.

 \subsection*{Acknowledgment}  The authors thank V. Zhelyabin for valuable  comments,  in particular, the authors learn Example~\ref{ex-gra} from V. Zhelyabin.

\newcommand{\noopsort}[1]{}


\begin{thebibliography}{10}

\bibitem{Albert} A. A. Albert, Power-associative rings, {\it Trans. Amer. Math. Soc.} {\bf 64} (1948) 552--593.

 

\bibitem{BN85} A.A. Balinskii, S.P. Novikov,   Poisson brackets of hydrodynamic type,   Frobenius algebras
and Lie algebras,   {\it Dokl. Akad. Nauk SSSR}    {\bf 283} (5) (1985)  1036--1039.


\bibitem{BJ} A. Bapat, D. Jordan, Lower central series of free algebras in symmetric tensor categories, {\it J. Algebra} {\bf 373} (2013) 299--311.

 


\bibitem{Burde} D. Burde,  W. de Graaf, Classification of Novikov algebras,  {\it Appl. Algebra Engrg. Comm. Comput.} {\bf 24} (2013) 1--15.


\bibitem{bd09} 
D. Burde, K. Dekimpe, K.  Vercammen, 
Complete LR-structures on solvable Lie algebras, 
{\it J. Group Theory}  {\bf 13} (5) (2010) 703-719.

\bibitem{pro-com} G. Deryabina, A. Krasilnikov, Products of commutators in a Lie nilpotent associative algebra, {\it J. Algebra}  {\bf 469} (2017) 84--95.

\bibitem{Dz09} A.S. Dzhumadildaev, Algebras with skew-symmetric identity of degree 3, 
{\it J. Math. Sci.} {\bf 161} (11) (2009) 11--30.

\bibitem{DIT11} A.S. Dzhumadildaev, N.A. Ismailov, K.M. Tulenbaev, Free bicommutative algebras, 
{\it Serdica Math. J.}  {\bf 37} (1) (2011) 25--44.

\bibitem{dzhuma}
A. Dzhumadildaev, K.  Tulenbaev, 
Bicommutative algebras, 
{\it Russian Math. Surveys} {\bf 58} (6) (2003) 1196–1197



\bibitem{dzhuma06}
A. Dzhumadildaev, K. Tulenbaev, 
Engel theorem for Novikov algebras,  
{\it Comm. Algebra} {\bf 34} (3) (2006) 883--888.


\bibitem{universal-Lie nilpotent} P. Etingof, J. Kim, X. Ma, On universal Lie nilpotent associative algebras, {\it J. Algebra} {\bf 321} (2009) 697--703.

\bibitem{fil}
V. Filippov, 
On right-symmetric and Novikov nil algebras of bounded index, 
{\it Math. Notes} {\bf 70} (1-2) (2001) 258--263.



\bibitem{GD79} I.M. Gelfand, I.Ya. Dorfman,  Hamiltonian operators and algebraic structures related to them, \textit{Funkts. Anal. Prilozhen} {\bf 13} (1979)  13--30. 


\bibitem{jen1} S. A. Jennings, Central chains of ideals in an associative ring, {\it Duke Math. J.} {\bf 9} (1942) 341--355.

\bibitem{ass4}
N. Ismailov, I. Kaygorodov, F.Mashurov,  
The algebraic and geometric classification of nilpotent assosymmetric algebras, 
{\it Algebr. Represent. Theory} {\bf 24} (1) (2021) 135--148. 



\bibitem{nov4}
I. Karimjanov, I. Kaygorodov, A. Khudoyberdiyev,  
The algebraic and geometric classification of nilpotent Novikov algebras, 
{\it J. Geom. Phys.} {\bf 143} (2019) 11--21.

\bibitem{bi4}
I. Kaygorodov, P. Páez-Guillán, V. Voronin, 
The algebraic and geometric classification of nilpotent bicommutative algebras, 
{\it Algebr. Represent. Theory} {\bf 23} (6) (2020) 2331--2347.




\bibitem{filtration} G. Kerchev, On the filtration of a free algebra by its associative lower central series, {\it J. Algebra} {\bf 375} (2013) 322--327.

\bibitem{Kl57} E. Kleinfeld,   Assosymmetric rings, Proc. Amer. Math. Soc. {\bf 8} (5) (1957) 983--986.



\bibitem{poi1} I.Z. Monteiro Alves, V.M. Petrogradsky, Lie structure of truncated symmetric Poisson algebras, {\it J. Algebra} {\bf 488} (2017) 244--281.


\bibitem{pr77} 
D. Pokrass, D. Rodabaugh, 
Solvable assosymmetric rings are nilpotent,
{\it Proc. Amer. Math. Soc.} {\bf 64} (1) (1977) 30--34.



\bibitem{Riley} D.M. Riley, PI-algebras generated by nilpotent elements of bounded index, {\it J. Algebra} {\bf 192} (1997)  1--13.

\bibitem{lie-solvable} R. K. Sharma, J. B. Srivastava, Lie solvable rings, {\it Proc. Amer. Math. Soc.} {\bf 94} (1985) 1--8.

\bibitem{poi2} S. Siciliano, Solvable symmetric Poisson algebras and their derived lengths, {\it J. Algebra} {\bf 543} (2020) 98--100.

\bibitem{poi3} S. Siciliano, H. Usefi, Solvability of Poisson algebras, {\it J. Algebra} {\bf 568} (2021), 349--361.


\bibitem{shest} 
I. Shestakov, Z.  Zhang, 
Solvability and nilpotency of Novikov algebras, 
{\it Comm. Algebra} {\bf 48} (12) (2020) 5412--5420.


\bibitem{TUZ} K. Tulenbaev, U. Umirbaev, V. Zhelyabin, On the Lie-solvability of Novikov algebras, 
{\it   J. Algebra Appl.} (2022) DOI: 10.1142/S0219498823501177.


\bibitem{zelm}
E. Zelmanov,  
A class of local translation-invariant Lie algebras, 
{\it Dokl. Akad. Nauk SSSR} 292 (6) (1987) 1294--1297.

\end{thebibliography}
\end{document}